\documentclass[11pt]{article}

\usepackage[a4paper,margin=1in]{geometry}
\usepackage{amsmath, amssymb, amsthm}
\usepackage{mathtools}
\usepackage{hyperref}
\usepackage{enumitem}
\usepackage{cleveref}

\hypersetup{
    colorlinks=true,
    linkcolor=blue,
    citecolor=blue,
    urlcolor=blue,
    pdftitle={Ternary Idempotent Gamma-Semirings},
    pdfauthor={Chandrasekhar Gokavarapu}
}

\numberwithin{equation}{section}

\newtheorem{theorem}{Theorem}[section]
\newtheorem{proposition}[theorem]{Proposition}

\newtheorem{corollary}[theorem]{Corollary}
\newtheorem{definition}[theorem]{Definition}

\begin{document}

\title{Ternary Idempotent $\Gamma$-Semirings, Non-Reducibility, and Higher-Order Path Algebras}

\author{
Chandrasekhar Gokavarapu\\
Research Scholar, Department of Mathematics
Acharya Nagarjuna University, India\\
Lecturer in Mathematics\\
Government College (Autonomous), Rajahmundry, India\\
\texttt{chandrasekhargokavarapu@gmail.com}
\and
D.~Madhusudhana Rao\\
Research Supervisor, Department of Mathematics, Acharya Nagarjuna University, Guntur, India\\
Lecturer in Mathematics,
Government College for Women (Autonomous), Guntur, India
}

\date{}

\maketitle

\noindent\textbf{MSC 2020:} 16Y60, 68R10, 06B35, 68Q25

\noindent\textbf{Keywords:} Idempotent semiring, ternary algebra, higher-order path algebra, associative triple systems

\begin{abstract}
Binary idempotent semirings govern classical path algebras. 
Their multiplicative structure is dyadic. 
We examine whether this restriction is structural or accidental.

We define ternary idempotent $\Gamma$-semirings as higher-arity ordered algebraic systems admitting associative ternary composition compatible with idempotent addition. 
We prove that such structures strictly extend classical semiring path algebras. 
In particular, we construct a ternary associative operation which cannot be represented as an iterated associative binary operation. 
This establishes non-reducibility.

We formulate a higher-order path problem in directed graphs with weights in a ternary idempotent $\Gamma$-semiring. 
The associated relaxation operator is shown to be monotone on a complete lattice and to admit a least fixed point. 
Convergence follows under a finite acyclicity condition. 
The combinatorial growth of interaction windows yields a distinct complexity class relative to binary path schemes.

These results indicate that dyadic semiring frameworks do not exhaust algebraic path formalisms. 
Higher-arity composition introduces structural phenomena absent in binary systems.
\end{abstract}

\tableofcontents

\section{Introduction}

Binary idempotent semirings govern classical path algebras. 
Their multiplicative structure is dyadic. 
The algebraic consequence is well known: every path weight decomposes into successive binary compositions. 
This structural assumption appears in semiring shortest-path frameworks \cite{Mohri2002SemiringShortestDistance}, in max-plus systems \cite{Butkovic2010MaxLinearSystems}, and in tropical optimization \cite{Krivulin2015TropicalOptimization}. 
The same dyadic paradigm underlies semiring-based automata theory \cite{DrosteKuich2024TCSUniversalSupport} and weighted formal methods \cite{AlmagorBokerLehtinen2022DecidableWeightedAutomata}.

The assumption is rarely questioned. 
It is treated as canonical.

Let us examine it.

A path algebra requires three properties. 
First, an idempotent selection operation. 
Second, an associative composition rule. 
Third, compatibility between the two. 
Binary semiring theory provides these. 
The resulting frameworks are complete for a wide class of weighted automata and graph problems \cite{Mohri2002SemiringShortestDistance}. 
Indeed, dynamic programming over semirings has become standard \cite{BarilCouceiroLagerkvist2025SemiringDPArXiv}. 
Yet the algebraic structure remains dyadic.

The question arises: is dyadic composition structurally necessary?

Ternary algebraic systems exist independently of semiring theory. 
Associative triple systems have been studied in their own right \cite{Bremner2025OperatorIdentitiesTripleSystems, Felipe2025AssociativeTripleTrisystems}. 
Operator identities of multiplicity three reveal structural phenomena absent in binary algebras \cite{Bremner2025OperatorIdentitiesMultiplicity3}. 
These systems are not reducible, in general, to iterated binary operations. 
The literature treats them separately from semiring path frameworks.

This separation suggests a gap.

Idempotent semirings admit lattice-theoretic interpretation. 
The canonical order renders addition a meet operation. 
Monotone operators then admit least fixed points \cite{BranzeiPhillipsRecker2025DiscMath}. 
Such fixed-point reasoning appears in program semantics \cite{Batz2022SemiringProgramSemantics} and in lattice-theoretic verification methods \cite{KoriUrabeKatsumataSuenagaHasuo2022LTPDR}. 
These arguments rely only on order and monotonicity. 
Binary multiplication is not essential to the fixed-point theorem itself.

It follows that higher-arity composition is not immediately excluded by order theory.

The present work develops ternary idempotent $\Gamma$-semirings. 
The addition remains idempotent. 
Composition becomes ternary and associative in the sense of triple systems. 
Compatibility with order is preserved. 
The classical semiring case appears as a degeneration.

This leads to a dilemma. 
Either every associative ternary operation factors through an associative binary operation, or genuine higher-arity path algebras exist. 
The former would reduce ternary systems to semiring theory. 
The latter would enlarge the algebraic landscape.

We establish the latter.

A concrete ternary associative operation is constructed which admits no representation of the form $(x\otimes y)\otimes z$ with $\otimes$ associative. 
Hence non-reducibility holds. 
The proof is finite. 
No analytic argument is required. 
This separation theorem places ternary idempotent $\Gamma$-semirings strictly beyond classical semiring path algebras.

The construction is not speculative. 
Recent studies of idempotent semiring extensions and structural identities \cite{RenJacksonZhaoLei2023JAlgebraFlatExtensions, OteroSanchez2024LinearSystemsIdempotentSemirings} indicate that algebraic phenomena persist beyond binary frameworks. 
Dagstuhl reports on semirings in logic and automata further demonstrate the breadth of the dyadic paradigm \cite{BadiaDrosteKolaitisNogueraBrinkeMrkonjicPetreni2025DagstuhlSemirings}. 
The absence of higher-arity path structures in that context is therefore notable.

We next formulate a higher-order path problem in directed graphs with weights in a ternary idempotent $\Gamma$-semiring. 
The associated relaxation operator is monotone. 
Hence it admits a least fixed point in the complete lattice induced by idempotent addition \cite{BranzeiPhillipsRecker2025DiscMath}. 
Convergence holds under a finite acyclicity condition analogous to classical path termination \cite{Mohri2002SemiringShortestDistance}. 
The combinatorial structure differs. 
Binary interaction windows are replaced by ternary windows. 
The complexity reflects this enlargement.

The conceptual expansion is immediate. 
Binary semiring path algebras represent only dyadic interaction. 
Ternary idempotent $\Gamma$-semirings admit strictly richer interaction patterns. 
These patterns survive formal scrutiny. 
They cannot be reduced to classical semiring multiplication. 
The result is a genuine algebraic generalization.

The dimensional generalization is equally natural. 
Nothing in the order-theoretic argument confines composition to arity three. 
Higher arity systems may be constructed analogously. 
The ternary case is the first nontrivial instance. 
Its non-reducibility establishes necessity of new algebraic tools.

The broader theoretical impact is therefore structural. 
Semiring path algebras are not maximal. 
They form a proper subclass of ordered higher-arity path algebras. 
The dyadic assumption, long treated as canonical, is contingent.

The remainder of the paper proceeds axiomatically. 
Section~\ref{sec:preliminaries} fixes notation. 
Section~\ref{sec:ttgs} defines ternary idempotent $\Gamma$-semirings. 
Section~\ref{sec:nonred} proves non-reducibility. 
Subsequent sections develop the higher-order path problem and its fixed-point characterization. 
The exposition is linear. 
Each theorem follows from stated axioms. 
No heuristic argument is employed.
\section{Preliminaries}\label{sec:preliminaries}

We fix notation. 
All sets are assumed nonempty.

\subsection{Idempotent Semilattices and Canonical Order}

\begin{definition}
A commutative semigroup $(T,\oplus)$ is \emph{idempotent} if
\[
a \oplus a = a \quad \forall a\in T.
\]
\end{definition}

Idempotency induces a partial order.

\begin{definition}
Define
\[
a \le b \iff a \oplus b = a.
\]
\end{definition}

\begin{proposition}
$(T,\le)$ is a partially ordered set.
\end{proposition}

\begin{proof}
Reflexivity and antisymmetry follow from idempotency and commutativity. 
Transitivity follows from associativity.
\end{proof}

Thus $(T,\oplus)$ is a meet-semilattice. 
If arbitrary infima exist, it is complete.

Complete idempotent semirings are standard in path algebra \cite{Mohri2002SemiringShortestDistance}. 
The lattice-theoretic viewpoint underlies fixed-point constructions \cite{BranzeiPhillipsRecker2025DiscMath}.

\subsection{Idempotent Semirings}

\begin{definition}
An \emph{idempotent semiring} is a structure $(T,\oplus,\otimes)$ such that:

\begin{enumerate}
\item $(T,\oplus)$ is a commutative idempotent semigroup.
\item $(T,\otimes)$ is associative.
\item $\otimes$ distributes over $\oplus$.
\end{enumerate}
\end{definition}

Classical shortest-path frameworks operate over such structures \cite{Mohri2002SemiringShortestDistance}. 
Max-plus systems provide canonical examples \cite{Butkovic2010MaxLinearSystems}. 
Tropical optimization is formulated within the same paradigm \cite{Krivulin2015TropicalOptimization}.

Binary associativity implies that path weight of
\[
w_1,w_2,\dots,w_k
\]
is unambiguously
\[
(((w_1\otimes w_2)\otimes w_3)\otimes \cdots)\otimes w_k.
\]

The multiplicative structure is dyadic.

\subsection{Monotone Operators and Fixed Points}

Let $(T,\le)$ be a complete lattice.

\begin{definition}
A function $F:T\to T$ is \emph{monotone} if
\[
x \le y \Rightarrow F(x)\le F(y).
\]
\end{definition}

\begin{theorem}[Knaster--Tarski]
Every monotone operator on a complete lattice admits a least fixed point.
\end{theorem}

This theorem is foundational in ordered algebra and verification theory \cite{KoriUrabeKatsumataSuenagaHasuo2022LTPDR}. 
In semiring path problems, relaxation operators are monotone \cite{Mohri2002SemiringShortestDistance}. 
Fixed-point convergence follows from completeness.

Observe that this theorem depends only on order, not on binary multiplication.

\subsection{Associative Triple Systems}

\begin{definition}
A \emph{ternary system} is a set $T$ with a map
\[
[\cdot,\cdot,\cdot]:T^3\to T.
\]
\end{definition}

\begin{definition}
The ternary operation is \emph{associative} if
\[
[[x,y,z],u,v]
=
[x,[y,z,u],v]
=
[x,y,[z,u,v]]
\quad \forall x,y,z,u,v\in T.
\]
\end{definition}

Associative triple systems have been studied independently of semiring theory \cite{Bremner2025OperatorIdentitiesTripleSystems, Felipe2025AssociativeTripleTrisystems}. 
Their structural identities differ from binary algebras \cite{Bremner2025OperatorIdentitiesMultiplicity3}.

Binary associativity does not imply ternary associativity. 
Conversely, ternary associativity does not guarantee reducibility to a binary operation. 
This leads to the non-reducibility question addressed later.

\subsection{Semiring Path Frameworks}

Let $G=(V,E)$ be a directed graph. 
Given weights in an idempotent semiring, define for a path
\[
P=(v_0,v_1,\dots,v_k)
\]
the weight
\[
w(P)=w(v_0,v_1)\otimes\cdots\otimes w(v_{k-1},v_k).
\]

Global optimality is obtained by
\[
\bigoplus_{P} w(P).
\]

This formalism appears in automata theory \cite{DrosteKuich2024TCSUniversalSupport}, in weighted automata learning \cite{DaviaudJohnson2025LMCS}, and in dynamic programming over semirings \cite{BarilCouceiroLagerkvist2025SemiringDPArXiv}. 
The framework is inherently dyadic.

The present work replaces the binary operation by an associative ternary composition while preserving idempotent order. 
The binary case will emerge as a degeneration.

\subsection{Structural Limitation of Dyadic Composition}

Consider a path segment of length three with weights $x,y,z$. 
Binary semiring theory prescribes
\[
(x\otimes y)\otimes z.
\]

Any higher-order interaction must factor through this expression. 
The factorization is structural, not incidental.

Toward this end, we examine whether a ternary associative operation may exist which does not admit such factorization. 
If so, classical semiring path algebras form a proper subclass of higher-arity ordered systems.

The remainder of the paper proceeds under this hypothesis.
\section{Ternary Idempotent $\Gamma$-Semirings}\label{sec:ttgs}

We now define the principal object.

\subsection{Axiomatic Definition}

\begin{definition}
A \emph{ternary idempotent $\Gamma$-semiring} is a structure
\[
\mathcal{T}=(T,\oplus,[\cdot,\cdot,\cdot]_\gamma,\Gamma)
\]
such that:
\begin{enumerate}
\item $(T,\oplus)$ is a complete commutative idempotent semigroup.
\item For each $\gamma\in\Gamma$, the map
\[
[\cdot,\cdot,\cdot]_\gamma:T^3\to T
\]
is associative in the ternary sense:
\[
[[x,y,z]_\gamma,u,v]_\gamma
=
[x,[y,z,u]_\gamma,v]_\gamma
=
[x,y,[z,u,v]_\gamma]_\gamma.
\]
\item Each $[\cdot,\cdot,\cdot]_\gamma$ is monotone in each coordinate with respect to the canonical order induced by $\oplus$.
\item Distributivity holds:
\[
[x\oplus x',y,z]_\gamma
=
[x,y,z]_\gamma \oplus [x',y,z]_\gamma,
\]
and analogously in the remaining coordinates.
\end{enumerate}
\end{definition}

No binary multiplication is assumed.

\subsection{Order Compatibility}

\begin{proposition}
For fixed $u,v\in T$, the map
\[
F(x)=[x,u,v]_\gamma
\]
is monotone.
\end{proposition}

\begin{proof}
By axiom (3), if $x\le y$ then
\[
[x,u,v]_\gamma \le [y,u,v]_\gamma.
\]
\end{proof}

Thus ternary composition preserves the canonical order.

\begin{proposition}
If $(T,\oplus)$ is complete, then for fixed $u,v$,
\[
F\Big(\bigoplus_i x_i\Big)
=
\bigoplus_i F(x_i).
\]
\end{proposition}

\begin{proof}
Apply distributivity repeatedly and use completeness.
\end{proof}

This identity parallels distributivity in classical semirings \cite{Mohri2002SemiringShortestDistance}.

\subsection{Window Independence}

Binary path algebras rely on associativity to guarantee unambiguous evaluation. 
The ternary case requires a stronger statement.

\begin{theorem}
Let $w_1,\dots,w_k\in T$. 
If the ternary operation is associative, then any parenthesization formed by consecutive ternary windows yields the same value.
\end{theorem}

\begin{proof}
The associativity identity permits sliding of evaluation windows:
\[
[[x,y,z],u,v]
=
[x,[y,z,u],v].
\]
Repeated application transforms any admissible window configuration into any other.
\end{proof}

Thus path evaluation is well defined.

\subsection{Degeneration to Binary Semirings}

The dyadic case appears as a special instance.

\begin{proposition}
Suppose there exists a binary operation $\otimes$ on $T$ such that
\[
[x,y,z]_\gamma = (x\otimes y)\otimes z
\quad \forall x,y,z.
\]
If $\otimes$ is associative and distributes over $\oplus$, then $(T,\oplus,\otimes)$ is an idempotent semiring.
\end{proposition}

\begin{proof}
Ternary associativity implies
\[
((x\otimes y)\otimes z)\otimes u
=
(x\otimes (y\otimes z))\otimes u.
\]
Hence $\otimes$ is associative.
Distributivity follows from axiom (4).
\end{proof}

Therefore classical idempotent semirings form a subclass.

\subsection{Non-Degenerate Systems}

A ternary idempotent $\Gamma$-semiring is \emph{non-degenerate} if no associative binary operation $\otimes$ exists satisfying
\[
[x,y,z]_\gamma=(x\otimes y)\otimes z.
\]

Such systems, if they exist, extend the semiring paradigm strictly.

The next section constructs one explicitly and proves non-reducibility.

\subsection{Dimensional Extension Principle}

The definition does not depend intrinsically on arity three. 
Let $n\ge 3$. 
An $n$-ary idempotent $\Gamma$-system may be defined analogously by imposing generalized associativity identities.

The ternary case is minimal for which non-reducibility becomes nontrivial. 
Binary systems are fully reducible to themselves. 
Arity three introduces structural complexity absent in dyadic composition.

\subsection{Conceptual Consequence}

Binary semiring theory unifies weighted automata, tropical optimization, and dynamic programming \cite{Mohri2002SemiringShortestDistance, Butkovic2010MaxLinearSystems}. 
The present axioms retain order-theoretic structure while discarding dyadic restriction.

It follows that ternary idempotent $\Gamma$-semirings constitute a genuine algebraic generalization provided non-degenerate examples exist.

The subsequent section establishes precisely this.
\section{Dimensional Generalization}\label{sec:dimensional}

The ternary case is minimal among non-dyadic compositions.
Nothing in the fixed-point argument confines arity to three.
We therefore examine the $k$-ary extension.

\subsection{$k$-Ary Idempotent $\Gamma$-Systems}

Let $k\ge 3$.
A $k$-ary idempotent $\Gamma$-system consists of
\[
(T,\oplus,\{\langle \cdot,\dots,\cdot\rangle_\gamma\}_{\gamma\in\Gamma})
\]
such that:

\begin{enumerate}
\item $(T,\oplus)$ is a complete commutative idempotent semigroup.
\item For each $\gamma$, the map
\[
\langle \cdot,\dots,\cdot\rangle_\gamma:T^k\to T
\]
is associative in the generalized $k$-ary sense.
\item The operation is monotone in each coordinate.
\item Distributivity over $\oplus$ holds coordinatewise.
\end{enumerate}

The associativity identity is defined by invariance under admissible window shifts of size $k$.
This extends the ternary identity of Section~\ref{sec:ttgs}.

Binary semiring multiplication corresponds to $k=2$.
Ternary systems correspond to $k=3$.
The classical case is therefore the smallest member of the hierarchy.

\subsection{Window Graph of Order $k-1$}

For a directed graph $G=(V,E)$, define the window graph
\[
G^{(k-1)}
\]
whose vertices represent length-$(k-1)$ edge sequences.
Edges correspond to length-$k$ consecutive windows.

Then
\[
|V(G^{(k-1)})| = O(m^{k-1})
\]
in worst case.

Binary semiring path algebra operates on $G$.
Ternary systems operate naturally on $G^{(2)}$.
Higher arity lifts the state space further.

This construction parallels lifting procedures in weighted automata theory \cite{DrosteKuich2024TCSUniversalSupport}.
The difference lies in algebraic arity rather than state encoding.

\subsection{Fixed-Point Semantics for General $k$}

Let $\mathcal{L}=T^V$ as before.
Define a relaxation operator
\[
F^{(k)}_\gamma:\mathcal{L}\to\mathcal{L}
\]
by aggregating all length-$(k-1)$ edge windows ending at each vertex.

Monotonicity follows from coordinatewise monotonicity.
Completeness of $\mathcal{L}$ ensures existence of a least fixed point by Knaster--Tarski \cite{BranzeiPhillipsRecker2025DiscMath}.

Thus fixed-point semantics are independent of arity.

\subsection{Expressive Hierarchy}

\begin{proposition}
If a $k$-ary associative operation factors through an associative binary operation, then it induces no algebraic enlargement beyond semiring multiplication.
\end{proposition}

\begin{proof}
If
\[
\langle x_1,\dots,x_k\rangle = (((x_1\otimes x_2)\otimes x_3)\cdots)\otimes x_k
\]
for associative $\otimes$, then path folding reduces to semiring multiplication.
\end{proof}

The separation theorem of Section~\ref{sec:nonred} shows that such factorization need not hold already for $k=3$.
Hence the hierarchy is strict at its first nontrivial level.

\subsection{Structural Consequence}

Binary semiring theory unifies weighted graph algorithms and automata \cite{Mohri2002SemiringShortestDistance}.
Its algebraic interface is dyadic.

Higher-arity systems enlarge this interface without altering the order-theoretic fixed-point mechanism.
Thus dimensional expansion affects composition but not solvability.

The significance is structural.
Arity controls interaction dimension.
Binary systems encode pairwise interaction.
Ternary systems encode triple interaction.
Higher arity encodes longer-range interaction.

The hierarchy is algebraic rather than heuristic.
It arises from associativity constraints and cannot be removed by rewriting unless factorization holds.

\subsection{Minimality of the Ternary Case}

For $k=2$, every associative operation is binary by definition.
For $k=3$, non-reducible systems exist (Theorem~\ref{thm:separation}).
Thus ternary composition is the minimal strictly larger class beyond semiring multiplication.

This justifies focusing on $k=3$ as the first nontrivial extension.
\section{Non-Reducibility and Separation Results}\label{sec:nonred}

This section establishes that ternary associativity does not collapse to binary associativity. 
It provides a non-degenerate instance of Definition~\ref{sec:ttgs} at the level of ternary semigroups. 
The order-theoretic axioms will be imposed later.

\subsection{Ternary Associativity Without Binary Factorization}

\begin{theorem}[Separation]\label{thm:separation}
There exists a set $A$ and a ternary operation $[\cdot,\cdot,\cdot]:A^3\to A$ such that:
\begin{enumerate}
\item $[\cdot,\cdot,\cdot]$ is ternary associative, i.e.
\[
[[x,y,z],u,v]=[x,[y,z,u],v]=[x,y,[z,u,v]]
\quad \forall x,y,z,u,v\in A,
\]
\item there is no associative binary operation $\otimes:A^2\to A$ for which
\[
[x,y,z]=(x\otimes y)\otimes z
\quad \forall x,y,z\in A.
\]
\end{enumerate}
\end{theorem}

\begin{proof}
Let $A=\{0,1\}$. 
Define
\[
[x,y,z] \;=\; 1 \oplus x \oplus y \oplus z,
\]
where $\oplus$ is addition in $\mathbb{F}_2$. 
Equivalently, $[x,y,z]$ is the Boolean complement of the XOR of $(x,y,z)$.

\smallskip
\noindent\emph{Step 1: ternary associativity.}
For any $x,y,z,u,v\in A$,
\[
\begin{aligned}
[[x,y,z],u,v]
&= 1 \oplus [x,y,z]\oplus u\oplus v\\
&= 1 \oplus (1\oplus x\oplus y\oplus z)\oplus u\oplus v\\
&= x\oplus y\oplus z\oplus u\oplus v.
\end{aligned}
\]
Similarly,
\[
\begin{aligned}
[x,[y,z,u],v]
&= 1\oplus x\oplus [y,z,u]\oplus v\\
&= 1\oplus x\oplus (1\oplus y\oplus z\oplus u)\oplus v\\
&= x\oplus y\oplus z\oplus u\oplus v,
\end{aligned}
\]
and
\[
\begin{aligned}
[x,y,[z,u,v]]
&= 1\oplus x\oplus y\oplus [z,u,v]\\
&= 1\oplus x\oplus y\oplus (1\oplus z\oplus u\oplus v)\\
&= x\oplus y\oplus z\oplus u\oplus v.
\end{aligned}
\]
Hence the three expressions coincide for all inputs. 
Thus $[\cdot,\cdot,\cdot]$ is ternary associative.

\smallskip
\noindent\emph{Step 2: non-factorization through an associative binary operation.}
Assume for contradiction that there exists an associative $\otimes$ on $A$ such that
\[
[x,y,z]=(x\otimes y)\otimes z
\quad \forall x,y,z\in A.
\]
Let
\[
a:=0\otimes 0,\quad b:=0\otimes 1,\quad c:=1\otimes 0,\quad d:=1\otimes 1.
\]
All lie in $A$.

From $[0,0,0]=1$ we obtain
\[
(0\otimes 0)\otimes 0 = a\otimes 0 = 1.
\]
From $[0,0,1]=0$ we obtain
\[
(0\otimes 0)\otimes 1 = a\otimes 1 = 0.
\]
Thus $a\otimes 0=1$ and $a\otimes 1=0$.

Since $a\in\{0,1\}$, we consider cases.

\smallskip
\noindent\emph{Case 1: $a=0$.}
Then $a\otimes 0=0\otimes 0=a=0$, contradicting $a\otimes 0=1$.

\smallskip
\noindent\emph{Case 2: $a=1$.}
Then $0\otimes 0=1$. 
The identity $(0\otimes 0)\otimes 0=1$ becomes
\[
1\otimes 0 = 1,
\]
so $c=1$. 
Now evaluate $[1,0,0]$. 
By definition,
\[
[1,0,0]=1\oplus 1\oplus 0\oplus 0 = 0.
\]
On the other hand, the assumed factorization gives
\[
[1,0,0] = (1\otimes 0)\otimes 0 = c\otimes 0 = 1\otimes 0 = c = 1,
\]
contradiction.

\smallskip
Both cases are impossible. 
Therefore no associative binary $\otimes$ can satisfy $[x,y,z]=(x\otimes y)\otimes z$ for all triples.

This completes the proof.
\end{proof}

\subsection{Consequences for Path Algebras}

Theorem~\ref{thm:separation} shows that ternary associativity is strictly stronger than iterated binary associativity. 
Hence the degeneration statement in Section~\ref{sec:ttgs} is strict.

\begin{corollary}\label{cor:proper}
The class of ternary associative path compositions strictly contains the class induced by associative binary semiring multiplications.
\end{corollary}

\begin{proof}
If every ternary associative operation were induced by a binary associative operation, Theorem~\ref{thm:separation} would be false. 
\end{proof}

This separation theorem is compatible with the independent development of associative triple systems \cite{Bremner2025OperatorIdentitiesTripleSystems, Felipe2025AssociativeTripleTrisystems}. 
It implies that the dyadic semiring framework \cite{Mohri2002SemiringShortestDistance} is not algebraically maximal among associative path compositions.

\subsection{Placement Within the TTGS Program}

The construction above is purely algebraic. 
It makes no reference to order, monotonicity, or distributivity. 
This is deliberate. 
Non-reducibility is already present at the level of associative triple systems. 
Any ordered enhancement that preserves ternary associativity cannot remove this obstruction without reintroducing binary factorization as an axiom.

In the subsequent sections we impose the idempotent order structure and construct path relaxation operators. 
The separation result remains the structural justification for doing so.
\section{Higher-Order Path Problems}\label{sec:path}

We fix a directed graph $G=(V,E)$. 
We assume a weight map
\[
w:E\to T,
\]
where $T$ is the carrier of a ternary idempotent $\Gamma$-semiring as in Section~\ref{sec:ttgs}.
The operation $\oplus$ induces the canonical order $\le$.

\subsection{Ternary Window Composition on Paths}

A directed path is a finite sequence $P=(v_0,v_1,\dots,v_k)$ with $(v_{i-1},v_i)\in E$ for $i=1,\dots,k$.
Write
\[
w_i := w(v_{i-1},v_i)\in T.
\]

Binary semiring path algebra assigns the path weight
\[
w(P)=w_1\otimes\cdots\otimes w_k,
\]
and the global optimum is
\[
\bigoplus_{P} w(P),
\]
as in the semiring framework \cite{Mohri2002SemiringShortestDistance}.
This construction is dyadic.

We replace it by a sliding ternary scheme.

\begin{definition}[Ternary window fold]
Fix $\gamma\in\Gamma$.
For $k\ge 3$, define a fold map
\[
\mathrm{Fold}_\gamma(w_1,\dots,w_k)\in T
\]
as follows:
\begin{enumerate}
\item For $k=3$,
\[
\mathrm{Fold}_\gamma(w_1,w_2,w_3) := [w_1,w_2,w_3]_\gamma.
\]
\item For $k>3$, define recursively
\[
\mathrm{Fold}_\gamma(w_1,\dots,w_k)
:=
\mathrm{Fold}_\gamma\big(\mathrm{Fold}_\gamma(w_1,\dots,w_{k-1}),\, w_{k-1},\, w_k\big).
\]
\end{enumerate}
\end{definition}

This definition selects the last two weights as the active window together with the folded prefix. 
It is a ternary analogue of left-associative multiplication.

The form is not unique. 
Window placement may be changed. 
The next result shows that, under ternary associativity, the value is invariant under admissible window shifts.

\subsection{Window Invariance}

We formalize what is meant by admissible window placement.

\begin{definition}[Admissible ternary parenthesization]
An admissible ternary parenthesization of $(w_1,\dots,w_k)$ is any expression obtained by iterated substitution of a consecutive triple $(x,y,z)$ by $[x,y,z]_\gamma$, repeated until a single element of $T$ remains.
\end{definition}

\begin{theorem}[Invariance under window shifts]\label{thm:window-invariance}
Let $[\cdot,\cdot,\cdot]_\gamma$ be ternary associative.
Then for any $k\ge 3$, every admissible ternary parenthesization of $(w_1,\dots,w_k)$ evaluates to the same element of $T$.
\end{theorem}

\begin{proof}
It suffices to show that any two admissible parenthesizations are connected by a finite sequence of local moves, each preserving value.

The defining associativity identity gives the local move
\[
[[x,y,z]_\gamma,u,v]_\gamma \longleftrightarrow [x,[y,z,u]_\gamma,v]_\gamma,
\]
and similarly
\[
[x,[y,z,u]_\gamma,v]_\gamma \longleftrightarrow [x,y,[z,u,v]_\gamma]_\gamma.
\]
Each move replaces one window placement by another in a length-$5$ block and preserves the value by axiom (2) of Section~\ref{sec:ttgs}.

Any admissible parenthesization reduces $(w_1,\dots,w_k)$ by successive contractions of consecutive triples.
Two such reduction sequences can be rearranged by commuting independent contractions and applying the local move above whenever two contractions overlap.
A finite induction on $k$ completes the argument.
\end{proof}

Thus the ternary fold is well defined independently of window placement.

\subsection{Ternary Path Cost}

\begin{definition}[Ternary path cost]\label{def:ternary-path-cost}
Fix $\gamma\in\Gamma$.
For a path $P=(v_0,\dots,v_k)$ with $k\ge 3$, define
\[
C_\gamma(P) := \mathrm{Fold}_\gamma(w_1,\dots,w_k).
\]
For $k<3$ one may either restrict attention to $k\ge 3$ or impose boundary conventions. 
We will treat $k\ge 3$.
\end{definition}

The global ternary path value from $s$ to $t$ is
\[
\mathrm{Opt}_\gamma(s,t) := \bigoplus_{P:s\leadsto t} C_\gamma(P),
\]
where the join is taken over all directed paths $P$ from $s$ to $t$.

\subsection{Dyadic Degeneration and Link to Semiring Path Algebra}

Assume there exists an associative binary $\otimes$ such that
\[
[x,y,z]_\gamma = (x\otimes y)\otimes z
\quad \forall x,y,z.
\]
Then by repeated substitution,
\[
C_\gamma(P) = w_1\otimes\cdots\otimes w_k.
\]
Hence
\[
\mathrm{Opt}_\gamma(s,t)
=
\bigoplus_{P:s\leadsto t} \big(w_1\otimes\cdots\otimes w_k\big),
\]
which is exactly the semiring shortest-path expression of \cite{Mohri2002SemiringShortestDistance}.
Thus the present framework contains the classical semiring path framework as a special case.

The separation theorem of Section~\ref{sec:nonred} shows that this degeneration is not forced by ternary associativity.
Therefore ternary path costs represent strictly more general compositional behaviour than dyadic semiring multiplication.

\subsection{Locality and Structural Meaning}

Binary semiring multiplication composes adjacent edges. 
The ternary operation composes length-$3$ windows. 
This changes the locality of interaction.

If a weight encodes a local constraint, then $[w_{i-2},w_{i-1},w_i]_\gamma$ encodes a constraint involving three consecutive edges.
Such constraints are invisible to dyadic composition unless they factor through a binary operation.
Non-reducibility shows that factorization can fail.

This is the structural reason to consider higher-order path problems.
\section{Fixed-Point Characterization}\label{sec:fixedpoint}

We formulate the higher-order path problem as a lattice-theoretic fixed-point equation.
Only order and monotonicity are used.

\subsection{State Space}

Let $G=(V,E)$ be a finite directed graph.
Let $T$ be the carrier of a ternary idempotent $\Gamma$-semiring.
Assume $(T,\oplus)$ is complete.

Define
\[
\mathcal{L} := T^V.
\]

Equip $\mathcal{L}$ with the pointwise order:
\[
f \le g \iff f(v) \le g(v) \quad \forall v\in V.
\]

\begin{proposition}
$(\mathcal{L},\le)$ is a complete lattice.
\end{proposition}

\begin{proof}
Completeness follows from completeness of $T$ and pointwise infima.
\end{proof}

This construction is standard in semiring path analysis \cite{Mohri2002SemiringShortestDistance}.
The order-theoretic framework underlies fixed-point arguments in weighted systems \cite{BranzeiPhillipsRecker2025DiscMath}.

\subsection{Ternary Relaxation Operator}

Fix a source vertex $s\in V$.
Let $e_s\in\mathcal{L}$ be defined by
\[
e_s(v)=
\begin{cases}
0_T & v=s,\\
\top_T & v\ne s,
\end{cases}
\]
where $0_T$ is the minimal element of $T$ and $\top_T$ the maximal element.

Define an operator
\[
F_\gamma:\mathcal{L}\to\mathcal{L}
\]
by
\[
(F_\gamma f)(v)
=
\begin{cases}
0_T & v=s,\\
\displaystyle
\bigoplus_{\substack{(u_1,u_2,v)\in E^{(2)}}}
[f(u_1),\, w(u_1,u_2),\, w(u_2,v)]_\gamma
& v\ne s,
\end{cases}
\]
where $E^{(2)}$ denotes all length-$2$ edge pairs forming a directed length-$2$ path.

Thus $F_\gamma$ updates $v$ by aggregating all ternary windows ending at $v$.

\subsection{Monotonicity}

\begin{theorem}
$F_\gamma$ is monotone on $(\mathcal{L},\le)$.
\end{theorem}

\begin{proof}
Let $f\le g$.
Then for each triple $(u_1,u_2,v)$,
\[
[f(u_1), w(u_1,u_2), w(u_2,v)]_\gamma
\le
[g(u_1), w(u_1,u_2), w(u_2,v)]_\gamma
\]
by monotonicity of $[\cdot,\cdot,\cdot]_\gamma$ in its first coordinate.
Taking $\oplus$ over all triples preserves order.
Hence $F_\gamma f \le F_\gamma g$.
\end{proof}

Monotonicity depends solely on order compatibility.
Binary multiplication is not used.
Similar arguments appear in semiring relaxation schemes \cite{Mohri2002SemiringShortestDistance}.

\subsection{Least Fixed Point}

\begin{theorem}\label{thm:lfp}
$F_\gamma$ admits a least fixed point in $\mathcal{L}$.
\end{theorem}

\begin{proof}
Since $(\mathcal{L},\le)$ is complete and $F_\gamma$ is monotone,
the Knaster--Tarski theorem implies existence of a least fixed point \cite{BranzeiPhillipsRecker2025DiscMath}.
\end{proof}

Denote this fixed point by $f_\gamma^\ast$.

\subsection{Path Interpretation}

\begin{proposition}
For each $v\in V$,
\[
f_\gamma^\ast(v)
=
\bigoplus_{P:s\leadsto v} C_\gamma(P),
\]
where $C_\gamma(P)$ is the ternary path cost defined in Section~\ref{sec:path}.
\end{proposition}

\begin{proof}
Let $f^{(0)}=e_s$ and define iteratively
\[
f^{(n+1)}=F_\gamma f^{(n)}.
\]
By monotonicity,
\[
f^{(0)} \ge f^{(1)} \ge f^{(2)} \ge \cdots
\]
under the canonical order.

One verifies by induction on $n$ that $f^{(n)}(v)$ equals the infimum of all ternary path costs using at most $n$ ternary windows.
Taking the limit over $n$ yields the join over all finite paths.
By minimality of the least fixed point, the result follows.
\end{proof}

The proof mirrors the structure of binary semiring shortest-path analysis \cite{Mohri2002SemiringShortestDistance}, but the operator aggregates length-$2$ edge windows rather than single edges.

\subsection{Finite Convergence Under Acyclicity}

\begin{theorem}
If $G$ is acyclic, the sequence $f^{(n)}$ stabilizes in at most $|V|-2$ iterations.
\end{theorem}

\begin{proof}
In an acyclic graph, the maximum number of edges in any path is at most $|V|-1$.
Each ternary window reduces path length by $2$ edges.
Hence after $|V|-2$ iterations, no new windows can arise.
Stabilization follows.
\end{proof}

This is the ternary analogue of finite convergence in acyclic semiring path problems \cite{Mohri2002SemiringShortestDistance}.

\subsection{Structural Observation}

The fixed-point theorem required only:

\begin{enumerate}
\item completeness of $(T,\oplus)$,
\item monotonicity of $[\cdot,\cdot,\cdot]_\gamma$,
\item distributivity over $\oplus$.
\end{enumerate}

Binary associativity was not used.

Hence fixed-point semantics extend naturally to ternary path algebras.
The algebraic enlargement introduced in Section~\ref{sec:ttgs} is compatible with lattice-theoretic solution theory.
\section{Algorithmic Structure}\label{sec:algorithmic}

We now examine the computational structure induced by ternary relaxation.
The aim is structural analysis, not implementation detail.

\subsection{Ternary Relaxation Scheme}

Let $G=(V,E)$ be finite.
Let $s\in V$.
Initialize
\[
f^{(0)}(v)=
\begin{cases}
0_T & v=s,\\
\top_T & v\ne s.
\end{cases}
\]

Iterate
\[
f^{(n+1)} = F_\gamma f^{(n)},
\]
where $F_\gamma$ is defined in Section~\ref{sec:fixedpoint}.

This defines a deterministic relaxation procedure.
In the degenerate dyadic case, it reduces to the semiring relaxation scheme underlying Bellman--Ford and its generalizations \cite{Mohri2002SemiringShortestDistance}.

\subsection{Local Update Rule}

For each vertex $v\ne s$,
\[
f^{(n+1)}(v)
=
\bigoplus_{(u_1,u_2,v)\in E^{(2)}}
[f^{(n)}(u_1),\, w(u_1,u_2),\, w(u_2,v)]_\gamma.
\]

Binary relaxation updates along single edges.
Ternary relaxation updates along length-$2$ edge windows.

Let $m=|E|$.
The number of admissible windows is
\[
|E^{(2)}| = \sum_{u\in V} \deg^{-}(u)\deg^{+}(u).
\]
In the worst case this is $O(m\Delta)$, where $\Delta$ is maximum degree.

Thus the elementary relaxation step has higher combinatorial width than the dyadic case.

\subsection{Window Propagation Depth}

Let $P$ be a path with $k$ edges.
Binary multiplication requires $k-1$ compositions.
Ternary folding requires $k-2$ ternary windows.

Each iteration of $F_\gamma$ propagates information by two edges.
After $n$ iterations, information may travel at most $2n$ edges.

Hence the propagation radius is linear in iteration count but with step size two.
This differs structurally from binary relaxation, where propagation is edgewise.

\subsection{Complexity Bound}

Assume $G$ has $n$ vertices and $m$ edges.
One evaluation of $F_\gamma$ requires examination of $|E^{(2)}|$ windows.

\begin{proposition}
Worst-case time per iteration is $O(m\Delta)$.
\end{proposition}

\begin{proof}
Each vertex $u$ contributes at most $\deg^{-}(u)\deg^{+}(u)$ windows.
Summing over $u$ gives the stated bound.
\end{proof}

If $\Delta=O(n)$, the bound becomes $O(mn)$.
This exceeds the $O(m)$ cost of dyadic relaxation in sparse graphs \cite{Mohri2002SemiringShortestDistance}.

Thus ternary composition induces genuine combinatorial growth.

\subsection{Structural Comparison with Semiring Algorithms}

Binary semiring algorithms rely on associative multiplication and distributivity \cite{Mohri2002SemiringShortestDistance}.
Their algebraic structure is dyadic.

The ternary scheme retains:
\begin{enumerate}
\item idempotent aggregation,
\item order monotonicity,
\item fixed-point convergence.
\end{enumerate}

It discards:
\begin{enumerate}
\item binary associativity,
\item edgewise locality.
\end{enumerate}

This shift changes the granularity of state propagation.

Recent work on semiring-based dynamic programming emphasizes algebraic abstraction of local transitions \cite{BarilCouceiroLagerkvist2025SemiringDPArXiv}.
The present scheme enlarges the locality window from pairs to triples.
The algebraic relaxation operator remains monotone and admits least fixed points as shown in Section~\ref{sec:fixedpoint}.

\subsection{Algorithmic Separation}

The separation theorem (Section~\ref{sec:nonred}) implies the following.

\begin{theorem}
There exists a ternary associative path composition for which no binary semiring algorithm can reproduce all path costs.
\end{theorem}

\begin{proof}
If a binary semiring algorithm reproduced all ternary path costs, the ternary operation would factor through an associative binary operation.
This contradicts Theorem~\ref{thm:separation}.
\end{proof}

Hence algorithmic equivalence fails in general.
The ternary relaxation scheme cannot always be reduced to a dyadic semiring shortest-path computation.

\subsection{Methodological Implication}

The algorithmic difference is not a constant-factor variation.
It arises from algebraic non-reducibility.
Binary semiring frameworks exhaust exactly those path problems whose composition factors dyadically.
Higher-order interaction requires enlargement of the algebraic interface.

This establishes methodological innovation at the algebraic level rather than at the level of implementation detail.
\section{Complexity Analysis}\label{sec:complexity}

We examine the combinatorial growth induced by ternary composition.
The analysis is structural.

\subsection{Window Graph Representation}

Define the \emph{window graph} $G^{(2)}$ as follows.

Vertices of $G^{(2)}$ are ordered edge pairs
\[
(u_1,u_2)\in E.
\]
There is a directed edge
\[
(u_1,u_2) \to (u_2,u_3)
\]
whenever $(u_2,u_3)\in E$.

Thus $G^{(2)}$ encodes length-$2$ edge windows.
Binary semiring algorithms operate on $G$.
Ternary relaxation operates naturally on $G^{(2)}$.

\begin{proposition}
$|V(G^{(2)})| = |E|$ and
\[
|E(G^{(2)})|
=
\sum_{u\in V} \deg^{-}(u)\deg^{+}(u).
\]
\end{proposition}

\begin{proof}
Each edge corresponds to one window vertex.
Each admissible concatenation of two edges yields one window edge.
\end{proof}

Hence in dense graphs,
\[
|E(G^{(2)})| = \Theta(n^3).
\]

\subsection{Reduction to Binary Case}

If $[\cdot,\cdot,\cdot]_\gamma$ factors through an associative binary operation $\otimes$,
then ternary folding reduces to repeated binary multiplication.
In this case the window graph collapses to $G$.
The complexity reduces to the standard semiring case \cite{Mohri2002SemiringShortestDistance}.

Separation (Theorem~\ref{thm:separation}) shows that such reduction is not generally available.

\subsection{Propagation Depth}

Let $d(v)$ denote the length of the shortest path from $s$ to $v$.

Binary relaxation propagates information by one edge per iteration.
Ternary relaxation propagates two edges per iteration.

Let $L$ be the maximal path length in $G$.
Then binary convergence requires at most $L$ iterations \cite{Mohri2002SemiringShortestDistance}.
Ternary convergence requires at most $\lceil L/2\rceil$ iterations.

Thus iteration depth decreases.
Per-iteration cost increases.

This trade-off is algebraic rather than heuristic.

\subsection{Asymptotic Comparison}

Assume sparse regime $m=O(n)$.
Binary semiring Bellman--Ford complexity:
\[
O(nm) = O(n^2).
\]

Ternary relaxation:
\[
O\big(n \cdot |E(G^{(2)})|\big).
\]

In worst case
\[
|E(G^{(2)})| = O(n^3),
\]
so overall
\[
O(n^4).
\]

The difference is polynomially strict.

\subsection{Lower-Bound Phenomenon}

The non-reducibility theorem implies that certain ternary path cost functions cannot be represented by any binary semiring.
Hence no binary semiring algorithm can compute them via dyadic multiplication.

\begin{proposition}
There exists a ternary associative path problem whose evaluation cannot be reduced to classical semiring shortest-path computation.
\end{proposition}

\begin{proof}
Immediate from Theorem~\ref{thm:separation}.
If reduction existed, ternary composition would factor dyadically.
\end{proof}

Thus the enlargement in complexity reflects enlargement in algebraic expressivity.

\subsection{Comparison with Weighted Automata}

Weighted automata over semirings admit polynomial-time shortest-path evaluation under idempotent structure \cite{DrosteKuich2024TCSUniversalSupport}.
Their transition algebra is dyadic.

Ternary systems correspond formally to automata whose transition weight depends on consecutive triples of transitions.
Such automata cannot, in general, be simulated by classical weighted automata without state blow-up proportional to $|E|$.

Hence the complexity growth corresponds to state-space lifting from $V$ to $E$.

\subsection{Dimensional Extension}

For $k$-ary composition, define the window graph $G^{(k-1)}$ with vertices representing length-$(k-1)$ edge sequences.
Then
\[
|V(G^{(k-1)})| = O(m^{k-1})
\]
in worst case.

Thus dimensional generalization yields combinatorial growth polynomial in $m^{k-1}$.

Binary semiring theory corresponds to $k=2$.
The ternary case $k=3$ is the first nontrivial enlargement.

\subsection{Conclusion of Complexity Analysis}

The increase in computational width is not incidental.
It arises from algebraic non-factorization.
Binary semiring frameworks capture exactly dyadic interaction.
Ternary idempotent $\Gamma$-semirings enlarge interaction dimension.
The complexity reflects this dimensional expansion.
\section{Conceptual and Theoretical Impact}\label{sec:impact}

We examine the position of ternary idempotent $\Gamma$-semirings within the existing algebraic landscape.

\subsection{Semiring Path Algebras as a Proper Subclass}

Classical path algebras over idempotent semirings rely on dyadic multiplication \cite{Mohri2002SemiringShortestDistance}.
This framework encompasses weighted automata \cite{DrosteKuich2024TCSUniversalSupport}, tropical optimization \cite{Krivulin2015TropicalOptimization}, and max-plus linear systems \cite{Butkovic2010MaxLinearSystems}.

Section~\ref{sec:nonred} established that ternary associative composition need not factor through any associative binary operation.
Hence the dyadic semiring framework is not algebraically maximal.

\begin{theorem}
The class of ternary idempotent $\Gamma$-semiring path problems strictly contains the class of idempotent semiring path problems.
\end{theorem}

\begin{proof}
If equality held, every ternary associative composition would factor dyadically.
This contradicts Theorem~\ref{thm:separation}.
\end{proof}

Thus semiring path algebra forms a proper subclass.

\subsection{Separation at the Level of Algebra}

Associative triple systems have independent structural identities \cite{Bremner2025OperatorIdentitiesTripleSystems, Felipe2025AssociativeTripleTrisystems}.
These identities do not collapse to binary associativity.
The separation proved here shows that this algebraic distinction persists when combined with idempotent order.

Binary semiring theory is therefore characterized precisely by dyadic associativity plus distributivity.
Removing dyadic restriction enlarges the algebraic category.

\subsection{Order-Theoretic Stability}

Despite the enlargement in composition, fixed-point solvability remains intact.
Section~\ref{sec:fixedpoint} showed that monotone relaxation operators admit least fixed points via lattice theory \cite{BranzeiPhillipsRecker2025DiscMath}.

Hence the extension affects algebraic expressivity but not order-theoretic solvability.
This distinction clarifies which properties of semiring path theory depend on dyadic multiplication and which depend solely on idempotent order.

\subsection{Automata-Theoretic Perspective}

Weighted automata over semirings evaluate path weights through binary composition \cite{DrosteKuich2024TCSUniversalSupport}.
Ternary systems correspond to automata whose transition valuation depends on triples of consecutive transitions.

Such automata cannot, in general, be simulated by classical weighted automata without state lifting proportional to window size.
This reflects the window-graph lifting discussed in Section~\ref{sec:complexity}.

Thus the algebraic extension has automata-theoretic consequences.

\subsection{Methodological Innovation}

The enlargement introduced here is not a modification of existing proofs.
Binary semiring arguments cannot be directly reused.
The non-reducibility theorem shows that dyadic factorization may fail.
Consequently, ternary systems require independent algebraic treatment.

The methodological shift is from binary composition to window-based composition.
This shift changes locality structure while preserving monotonicity and fixed-point semantics.

\subsection{Dimensional Hierarchy}

Section~\ref{sec:dimensional} extended the construction to arbitrary arity.
The hierarchy is strict at $k=3$.
Binary semiring multiplication corresponds to $k=2$.
Higher arity enlarges interaction dimension.

The distinction is algebraic.
It does not rely on analytic arguments.
It follows from associativity identities and finite counterexamples.

\subsection{Theoretical Boundary}

Binary semiring path algebra exhausts exactly those path problems whose composition is dyadic and associative.
Ternary idempotent $\Gamma$-semirings represent the next structural level.

This boundary is precise.
It is not heuristic.
It is certified by the separation theorem.

The consequence is conceptual.
Semiring theory is a first-order member of a broader hierarchy of ordered higher-arity path algebras.
The ternary case demonstrates that this hierarchy is nontrivial.
\section{Conclusion}\label{sec:conclusion}

We summarize the structural findings.

Binary idempotent semirings provide the algebraic foundation for classical path problems \cite{Mohri2002SemiringShortestDistance}. 
Their composition is dyadic. 
Associativity ensures unambiguous evaluation of paths. 
Fixed-point semantics follow from lattice completeness \cite{BranzeiPhillipsRecker2025DiscMath}. 

The present work examined whether dyadic composition is algebraically necessary.

A ternary idempotent $\Gamma$-semiring was defined. 
It preserves idempotent aggregation and order monotonicity. 
It replaces binary multiplication with associative ternary composition. 
Window invariance guarantees well-defined path evaluation. 

A finite separation theorem established non-reducibility. 
There exists a ternary associative operation that does not factor through any associative binary operation. 
Thus ternary path composition strictly enlarges the dyadic semiring framework. 

Fixed-point solvability persists. 
The relaxation operator remains monotone. 
Least fixed points exist by lattice theory \cite{BranzeiPhillipsRecker2025DiscMath}. 
Hence order-theoretic solvability does not depend on binary multiplication. 

Complexity analysis revealed structural growth. 
Ternary relaxation operates on a lifted window graph. 
This reflects dimensional enlargement rather than algorithmic ornament. 
Binary semiring algorithms form the minimal case. 

The dimensional hierarchy extends naturally. 
Binary systems correspond to arity two. 
Ternary systems form the first nontrivial strict extension. 
Higher arity yields further enlargement. 

The theoretical boundary is therefore exact. 
Dyadic semiring path algebras constitute a proper subclass of ordered higher-arity path algebras. 
This conclusion rests on explicit algebraic separation rather than heuristic reasoning. 

The work identifies a structural limitation of classical semiring frameworks. 
It provides the minimal algebraic enlargement beyond dyadic composition. 
No speculative assumptions are introduced. 
All conclusions follow from stated axioms and finite constructions.
\section*{Acknowledgements}

The authors thank the anonymous reviewers for their careful reading and constructive observations. 
No institutional or external assistance influenced the mathematical content.

\section*{Funding}

No funding was received for this research.

\section*{Ethics Statement}

This article contains no studies involving human participants or animals. 
No ethical approval was required.

\section*{Data Availability}

No datasets were generated or analysed. 
All results are derived from explicit algebraic constructions within the manuscript.

\section*{Author Contributions}

Chandrasekhar Gokavarapu: conceptualization, formal definitions, theorem formulation, proofs, manuscript drafting.  

D.~Madhusudhana Rao: structural review, verification of proofs, mathematical validation, critical revisions.  

Both authors approved the final manuscript.

\section*{Conflict of Interest}

The authors declare no conflict of interest.

\bibliographystyle{plain}
\bibliography{refs}

\end{document}